\documentclass[12pt,reqno]{amsart}%
\usepackage{amsmath, color}
\usepackage[margin=1.5in]{geometry}
\usepackage{graphicx}
\usepackage{amsfonts}
\usepackage{amssymb}%
\usepackage[utf8]{inputenc}
\usepackage{comment}
\usepackage{hyperref}
\usepackage{mathtools}
\usepackage{xcolor}

\newtheorem{theorem}{Theorem}[section]
\theoremstyle{plain}

\newtheorem{lemma}[theorem]{Lemma}

\newtheorem{proposition}[theorem]{Proposition}

\numberwithin{equation}{section}
\theoremstyle{definition}
\newtheorem{definition}{Definition}
\theoremstyle{remark}
\newtheorem{remark}{Remark}[section]
\allowdisplaybreaks

\def\re{\mathbb{R}}

\newcommand{\C}{\mathbb{C}}
\newcommand{\R}{\mathbb{R}}

\title[HSL graphs]{Optimal Regularity for Hölder continuous Hamiltonian Stationary Lagrangian graphs 
}

\author{Arunima Bhattacharya and W. Jacob Ogden}

\address{Department of Mathematics, Phillips Hall\\
 The University of North Carolina at Chapel Hill, NC 27599 }
\email{arunimab@unc.edu}

\address{Department of Mathematics, Padelford Hall\\
University of Washington, Seattle, WA 98195
}
\email{wjogden@uw.edu}

\begin{document}

\begin{abstract}
In this paper, we establish optimal regularity for Hölder continuous Hamiltonian stationary Lagrangian graphs in $\mathbb{C}^n$. We prove that such a graph is smooth whenever its Hölder exponent is strictly larger than $\frac{1}{3}$ and the Lagrangian phase is supercritical, which yields semi-convexity of the potential. We establish the optimality of our result by constructing explicit singular solutions to the fourth order Hamiltonian stationary equation when the Hölder exponent of the graph is $\frac{1}{3}$. The singular solutions exist even under the strongest convexity assumption on the Lagrangian phase, namely the hypercritical phase, which enforces convexity of the potential. This presents a striking departure from the theory of special Lagrangian graphs.

\end{abstract}

\maketitle

\section{Introduction}

The theory of special Lagrangian graphs in $\mathbb{C}^n$ has been extensively studied with many deep and elegant results. In contrast, the corresponding fourth order generalization, namely Hamiltonian stationary Lagrangian graphs, remains largely unexplored. In this work, we show that fundamental regularity properties of solutions to the special Lagrangian equation break down entirely in the Hamiltonian stationary setting. At the same time, we identify sharp conditions under which comparable regularity results continue to hold.

In $\mathbb{C}^{n}= \R^n \times \R^n$, with the standard K\"ahler structure, the Lagrangian phase or angle, $\Theta$, of the Lagrangian submanifold $L_u=\{(x,Du)\:\vert\ x \in \Omega\subseteq \mathbb{R}^n\}$, can be expressed as 
\[
\Theta=\sum_{i=1}^n \arctan \lambda_i
\]
where $\lambda_i$'s are the eigenvalues of the Hessian of $u$ and $\Omega$ is a fixed bounded domain in $\R^n $. The mean curvature vector along
$L_u$ can be written as
\[
\vec{H}=J\nabla_g\Theta
\]
where $\nabla_g$ is the gradient operator of $L_{u}$ (see the work of Harvey-Lawson \cite[(2.19)]{HL}) and $g$ is the induced metric from the Euclidean metric
on $\mathbb{C}^{n}$, which can be written as
\[
g=I_n+(D^{2}u)^{2}.
\]
We say that $u$ solves the \textit{Hamiltonian stationary equation}  when the Lagrangian phase $\Theta$ is $g$-harmonic (cf. Oh \cite{Oh}, Schoen-Wolfson \cite[Proposition 2.2]{SW03}):
\begin{equation}
\Delta_g\Theta=0 \label{hstat0}.
\end{equation}
The gradient graph of $u$ is a Hamiltonian stationary Lagrangian submanifold, which is a critical point of the volume functional under Hamiltonian deformations, that is, variations generated by $J\nabla_g \eta$ for compactly supported smooth functions $\eta$.

When $\Theta$ is constant, $u$ solves the \textit{special Lagrangian equation}\begin{equation}
   \Theta=c. \label{slag}
\end{equation}
The gradient graph of $u$ is a special Lagrangian submanifold, which is a critical point of the volume functional under all compactly supported variations of the submanifold $L_u$.

In this paper, we resolve the question of regularity of H\"older continuous Hamiltonian stationary Lagrangian graphs. Before we state our main results, we first clarify the notion of a weak solution to the Hamiltonian stationary equation.

\begin{definition} \label{def1}
    We say that $u \in C^1( \Omega)$ is a weak solution of \eqref{hstat0} if 
\begin{enumerate} 
\item $u \in W^{2,1}( \Omega)$,
\item $\sqrt{  g }  \in L^1(\Omega)$,
\item $\Theta \in W^{1,2} ( \Omega)$, and
\item for all compactly supported smooth functions $\eta $ on the cylinder $\Omega \times \R^n$, $u$ satisfies

\[\displaystyle \int  \langle \nabla_g \Theta , \nabla_g \eta\rangle \: d (\mathcal{H}^n  \llcorner L_u) =0.\]  
\end{enumerate}
\end{definition}

Note that, under the above conditions, $\nabla_g \eta$ is well defined a.e. by projecting the Euclidean gradient onto the tangent plane of $L_u$.

\medskip
Our first result is the following.

\begin{theorem} \label{Thm_reg_holder}
Let $u\in C^{1, \beta}$ be a weak solution of \eqref{hstat0} on the unit ball $B_1\subset \mathbb{R}^n$  with $\beta \in ( \frac13,1) $ and $|\Theta| \geq (n-2)\frac{\pi}{2}+\delta$. Then $u$
is smooth in $B_{1/2}$ and satisfies the estimate
    \begin{equation*}
        \|u\|_{C^{k, \alpha}(B_{1/2})}\leq C(\alpha,k,  n,  \delta, \|u\|_{C^{1, \beta}(B_{1})} ) 
        \end{equation*} where $k\geq 2$.
\end{theorem}

We prove the sharpness of the above Hölder exponent $\beta$ by constructing the following counterexamples, which hold even under the strongest convexity condition on the Lagrangian phase, namely, $|\Theta|>(n-1)\frac{\pi}{2}$.

\begin{theorem}\label{Thm_sing_soln} On $B_1 \subset \R^n$ $(n \geq 3)$, for each positive integer $k$ there exists a weak solution $u^{(n,k)}$ to \eqref{hstat0} with $|\Theta|> (n-1) \frac \pi 2$
such that
$u^{(n,k)} \in  C^{1, \frac 1{2k+1}} \setminus C^{1,\frac1{2k+1} + \varepsilon }$ for all $\varepsilon >0$. 
\end{theorem}

\begin{remark}
    The above result shows that a convex solution $u\in C^{1, \frac{1}{3}}$ to the Hamiltonian stationary equation need not be any more regular. This is a notable departure from the theory of special Lagrangians, for which Chen-Shankar-Yuan \cite{CSY} established that convex viscosity solutions of the special Lagrangian equation must always be smooth. 
\end{remark}

\begin{remark}
    The singular solutions constructed in this work have gradient graphs that are smooth submanifolds of $\mathbb{C}^n$. It is an interesting question whether there exists a Hamiltonian stationary Lagrangian graph with a hypercritical phase or with a convex potential that has a genuine geometric singularity, i.e., a gradient graph with unbounded curvature. 
\end{remark}

 In the case when $\beta$ of Theorem \ref{Thm_reg_holder} is 1, i.e., $u\in C^{1,1}(B_1)$, we obtain a stronger regularity result.

\begin{theorem}\label{Thm_reg_Lipschitz}
Let  $u \in C^{1,1}$ be a weak solution of \eqref{hstat0} on the unit ball $B_1 \subset \mathbb{R}^n$. Then $u$ is smooth on some open set $U_0$ containing 
\[
  \left  \{\, x \in B_{\frac12} \: \vert \: |\Theta(x)| \geq (n-2)\frac{\pi}{2} \,\right \},
\]
and satisfies the estimate
\begin{equation*}
   \|u\|_{C^{k,\alpha}(B_r)} \;\leq\; C\big(\alpha, k, n, r, \|D^2 u\|_{L^\infty(B_{2r})}\big),
\end{equation*}
where $k \geq 2$ and $B_{2r} \subset U_0$.
\end{theorem}

We emphasize that we do not assume that $|\Theta| \geq (n-2)\frac{\pi}{2}$ holds everywhere in $B_1$. Instead, we show that our result holds on an open set $U_0$ that contains
$\{\, x \in B_{\frac12} \: \vert \: |\Theta(x)| \geq (n-2)\frac{\pi}{2}\,\}$. We elaborate on this point further in the discussion in Section \ref{disc}.
\begin{remark}
    When $n\leq 4$, any $C^{1,1}$ solution of \eqref{hstat0} is smooth as shown by Bhattacharya in \cite{AB2025}. When $n > 4$, let $\lambda_{\min}$ and $\lambda_{\max}$ denote the smallest and largest eigenvalues of the Hessian. If these satisfy $\arctan(\lambda_{\min}) > (\Theta-\pi)/n$ or $\arctan(\lambda_{\max}) < (\Theta+\pi)/n$, then the conclusion of Theorem \ref{Thm_reg_Lipschitz} follows  from using a rigidity result of Ogden-Yuan \cite[Corollary 1.1]{OY} and \cite{AB2025}.

    It is worth noting that in higher codimensions Lawson-Osserman \cite{LO} constructed examples of Lipschitz minimal submanifolds, including graphical ones, that fail to be $C^{1}$. Combined with the absence of a rigidity result for special Lagrangian cones in dimensions $n>4$, this indicates that imposing restrictions on the phase, or on the minimal or maximal eigenvalues of the Hessian of $u$, may be essential in establishing higher regularity for Lipschitz continuous Hamiltonian stationary Lagrangian graphs.
\end{remark}

\begin{remark}
    We would also like to point out that without the condition $|\Theta| \geq (n-2)\frac{\pi}{2}+\delta$, the result of Theorem \ref{Thm_reg_holder} remains valid provided that one assumes $\arctan(\lambda_{\min}) > (\Theta-\pi)/n+\delta$. This follows from \cite{OY} and \cite{AB2025}. 
\end{remark}

To highlight the distinction between our results and those established in special Lagrangian geometry, we first briefly describe how the range of the Lagrangian phase, $\Theta$, affects the concavity properties of the arctangent operator, which plays a key role in regularity. The phase $(n-2)\frac{\pi}{2}$ is referred to as critical: indeed, the level set $\{ \lambda=(\lambda_1,..,\lambda_n) \in \mathbb{R}^n \: \vert \:  \Theta=\sum_{i=1}^n \arctan \lambda_i\}$ is concave precisely when $|\Theta|\geq (n-2)\frac{\pi}{2}$ \cite[Lemma 2.2]{YY06}. When $|\Theta|\geq (n-1)\frac{\pi}{2}$, the phase is termed hypercritical. In this range, the concavity of the level set is immediate, as the condition enforces $\lambda > 0$, making the arctangent operator concave. Finally, when $|\Theta|\geq (n-2)\frac{\pi}{2}+\delta$, the phase is called supercritical. In this range, one can show that the Hessian of the potential $u$ admits a uniform negative lower bound dependent on $\delta$; as a consequence, the potential is semi-convex.

For the special Lagrangian equation \eqref{slag} with phase $|\Theta|\geq (n-2)\frac{\pi}{2}$, Warren-Yuan \cite{WY9,WY} and Wang-Yuan \cite{WaY} established Hessian estimates showing that even $C^{0}$ viscosity solutions are in fact analytic. When passing to the Hamiltonian stationary setting, our results reveal the following:  Theorem \ref{Thm_reg_holder} demonstrates that one can still expect full regularity for $C^{1,\frac{1}{3}+}$ solutions to \eqref{hstat0} provided the phase satisfies $|\Theta|\geq (n-2)\frac{\pi}{2}+\delta$. However, Theorem \ref{Thm_sing_soln} shows that singular $C^{1,\frac{1}{3}}$ solutions to \eqref{hstat0} exist even under the strongest convexity assumption on the phase $|\Theta|> (n-1)\frac{\pi}{2}$. This stands in sharp contrast with the special Lagrangian equation, where singular solutions arise only in the subcritical range $|\Theta|<(n-2)\frac{\pi}{2}$, that is, when the arctangent operator loses all convexity properties. In this case, Nadirashvili-Vl\u{a}du\c{t} \cite{NV} constructed $C^{1,\frac{1}{3}}$ solutions in dimension three, while Wang-Yuan \cite{WangY} proved the existence of solutions $u \in C^{1,\frac{1}{2m+1}} \setminus C^{1,\frac{1}{2m+1}+\varepsilon}$ for every integer $m \geq 1$ in dimension three.

Hamiltonian stationary submanifolds can be viewed as a natural extension of minimal Lagrangian submanifolds in Kähler-Einstein manifolds; in particular, they include special Lagrangians arising in Calabi-Yau manifolds. The existence and stability problem has been studied by
many authors via different approaches (cf. \cite{MR1062973}, \cite{MR1611051},
\cite{SWJDG}, \cite{HeleinRomon2002}, \cite{MR1986315}, \cite{HeleinRomon2005}, \cite{JLS}, and references therein). Every special Lagrangian submanifold is Hamiltonian stationary. However, the converse does not hold: for example, the Clifford torus in $\mathbb{C}^2$ is Hamiltonian stationary but not special Lagrangian. Moreover, there exist non-flat cones that are Hamiltonian stationary without being special Lagrangian (see the work of Schoen-Wolfson \cite{SW03}). For a recent survey of progress on regularity theory for special Lagrangian and Hamiltonian stationary equations, we refer the reader to \cite{chen2022regularity}.

Another source of difficulty in the study of Hamiltonian stationary submanifolds arises from the fact that the maximum principle, a key tool used in the study of minimal surface equations, does not apply to the fourth order equations governing these submanifolds. Moreover, even in the simplest setting of $\mathbb{C}^2$ equipped with the standard symplectic form, Hamiltonian stationary submanifolds lack a monotonicity property (see the work of Minicozzi \cite{minicozzi1995willmore} and Schoen-Wolfson \cite{SWJDG}), creating significant challenges in understanding singularities. Recent progress in this direction can be found in the works of Pigati-Rivi\`ere \cite{pigati2024variational} and Orriols \cite{orriols2024existence}.

In two dimensions, Schoen-Wolfson \cite{SW03} established regularity results in the setting of general Kähler manifolds, where singularities are known to occur; such singularities are non-graphical (see \cite[Section 7]{SWJDG}).
 In the Euclidean setting, they showed that any $C^{2,\alpha}$ solution of \eqref{hstat0} is smooth \cite[Proposition 4.6]{SWJDG}. In \cite{ChenWarren}, Chen-Warren proved that weak  $C^{1,1}$ solutions
to  \eqref{hstat0} are smooth as long as one of these conditions is met: The Lagrangian phase, $\Theta$, satisfies $|\Theta|\geq (n-2)\frac{\pi}{2}+\delta$; the potential $u$ is strongly convex; or the Hessian of $u$ is bounded by a constant sufficiently less than 1. Smoothness of $C^{1,1}$ solutions to \eqref{hstat0} without restrictions on the phase or the Hessian was shown in two dimensions by Bhattacharya-Warren \cite{BW2} and in dimensions up to four by Bhattacharya \cite{AB2025}. The smoothness of $C^1$ regular Hamiltonian stationary submanifolds of a symplectic manifold was established by Bhattacharya-Chen-Warren \cite{BCW} through an analysis of the double divergence form of the fourth order nonlinear scalar equation arising from studying the stationary points of
the volume of $L_u=(x,Du(x))$ in an open ball $B\subset\mathbb{R}^{2n}$ equipped
with a Riemannian metric among nearby competing gradient graphs $L_{t}=(x, Du(x)
+t D\eta(x))$ for compactly supported smooth functions $\eta$. This fourth order double divergence form of the Hamiltonian stationary equation was also studied by Bhattacharya-Skorobogatova \cite{BhaSko} to prove a partial regularity result for Lipschitz continuous Hamiltonian stationary Lagrangian graphs in all dimensions. For results on the compactness of Hamiltonian stationary submanifolds in $\mathbb{C}^n$ and more general symplectic manifolds, based on curvature and smoothness estimates derived from the regularity theory of the governing fourth order equations, we refer the reader to the works of Chen-Warren \cite{chen2024compactification} and Chen-Ma \cite{chen2024compactness}.

\medskip
The regularity of Hölder continuous Hamiltonian stationary Lagrangian graphs has long been an open problem. By developing new techniques, inspired in part by the methods of Wang-Yuan \cite{WangY}, we resolve this question. Let us now provide a brief overview of the main ideas used to prove our results. Note that throughout this paper, we assume $\Theta > 0$, since the case $\Theta < 0$ can be handled analogously by symmetry. 

\subsection{Main ideas: H\"older continuous HSL graphs} To prove Theorem \ref{Thm_reg_holder}, we perform a downward rotation of the Hamiltonian stationary Lagrangian graph with potential $u\in C^{1, \frac{1}{3}+}$, which is based on the rotation developed by Yuan \cite{YY02}.
We perform the rotation in such a way that the rotated submanifold is another Hamiltonian stationary Lagrangian graph, which we show to be smooth via a result of Chen-Warren \cite{ChenWarren}. 
Using the fact that smooth graphs lie within a distance strictly less than $Cr^2$ from their tangent planes in a small ball of radius $r$ about the tangent point, we perform a careful upward rotation of the rotated gradient graph on a ball of radius $\frac{r}{2}$. Under the assumption that the original potential is not in $C^{1,1}$, one then obtains a point at which the tangent plane to this gradient graph is vertical. However, we prove that this contradicts the Hölder continuity of the original Hamiltonian stationary Lagrangian graph for exponents larger than $\frac{1}{3}$.

\subsection{Main ideas: singular solutions} 
Our proof of Theorem \ref{Thm_sing_soln}, which constructs singular solutions $u \in  C^{1, \frac 1{2k+1}} \setminus C^{1,\frac1{2k+1} + \varepsilon }$ to \eqref{hstat0}, draws inspiration from the construction of singular solutions to \eqref{slag} by Nadirashvili-Vl\u{a}du\c{t} \cite{NV} and Wang-Yuan \cite{WangY}. The idea behind our proof is the following: the first step is to build an analytic solution on $\R^n$ with carefully chosen properties, using the Cauchy-Kovalevskaya Theorem. To generate a singularity, we design the solution so that its graph can be rotated in such a way that the tangent plane at one point becomes vertical, while still remaining a graph. This is achieved by prescribing Cauchy data so that the largest eigenvalue of the Hessian attains a strict local maximum at the origin. We use rotationally symmetric Cauchy data to simplify the analysis of the largest eigenvalue. Similar rotation techniques were recently used by Bhattacharya-Shankar \cite{BS1} to construct singular solutions of the Lagrangian mean curvature flow, and by Mooney-Savin \cite{MooneySavin24} to construct new singular solutions of the minimal surface system.

\subsection{Main ideas: Lipschitz regular HSL graphs}  

We prove Theorem \ref{Thm_reg_Lipschitz} by introducing an upward rotation that rotates a Lipschitz continuous gradient graph $L_u = (x, Du(x))$ by a small angle, $\gamma$, so that, in the new coordinates $(\bar{x}, \bar{y})$, it remains a Lipschitz continuous gradient graph of a potential $\bar{u}(\bar{x})$, whose phase lies in the supercritical range. This method requires only that $u \in C^{1,1}$, without requiring convexity constraints on $u$ or that $u$ solves any equation. In the rotated coordinates, the arctangent operator can be modified to a concave one, which yields higher regularity for $\bar{u}$. Rotating back to the original coordinates then implies smoothness of the original potential.

\subsubsection{Discussion} \label{disc} We explain that our argument used to prove Theorem \ref{Thm_reg_Lipschitz} does not rely on the assumption $\Theta \geq (n-2)\frac{\pi}{2}$ throughout $B_1$. Rather, the result holds on an open set $U_0$ that contains
$\{\, x \in B_1 \:\vert\:  \Theta(x) \geq (n-2)\frac{\pi}{2} \,\}$. To see why, suppose that $\Theta \geq (n-2)\frac{\pi}{2}$ everywhere. The strong minimum principle then tells us that $\Theta$ can reach the critical phase only on the boundary of $B_1$, unless $u$ itself is a special Lagrangian solution. In this situation, the Harnack inequality provides a quantitative strengthening: inside a slightly smaller ball $B_{0.9}$ we actually obtain $\Theta > (n-2)\frac{\pi}{2} + \delta$. At this stage, \cite[Theorem 1.2]{ChenWarren} applies directly, which shows that the desired regularity follows even without using the upward rotation argument.

On the other hand, one advantage of working with $C^{1,1}$ regularity is that it already suffices to carry out the upward rotation. This allows us to treat solutions whose phase may in fact cross the critical value. The effect of the upward rotation is to increase the phase by $n\gamma$. As a result, the rotated Hamiltonian stationary Lagrangian graph is smooth wherever $\Theta + n\gamma \geq (n-2)\frac{\pi}{2} + \delta$, so the original graph was smooth wherever $\Theta \geq (n-2) \frac \pi 2 - n \gamma + \delta $.

\subsection*{Organization}The organization of the paper is as follows. In Section \ref{Sec_Holder}, we present a detailed study of the geometric properties of Hölder continuous Hamiltonian stationary Lagrangian graphs and prove Theorem \ref{Thm_reg_holder}. In Section \ref{Sec_singsoln}, we show the sharpness of the result in Theorem \ref{Thm_reg_holder} by constructing singular solutions to \eqref{hstat0}, proving Theorem \ref{Thm_sing_soln}. In Section \ref{Sec_Lipschitz}, we describe an upward rotation technique, which we use to prove Theorem \ref{Thm_reg_Lipschitz}.

\medskip

 \subsection*{Acknowledgments} The authors thank Ravi Shankar, Micah Warren, and Yu Yuan for helpful comments. 
AB acknowledges
the support of NSF grant DMS-2350290, the Simons Foundation grant MPS-TSM-00002933, and a Bill Guthridge fellowship from UNC-Chapel Hill. WJO acknowledges the support of the NSF Graduate Research Fellowship Program under grant DGE-2140004 and NSF grant DMS-2453862.

\section{Regularity for H\"older continuous HSL graphs}
\label{Sec_Holder}

In this section, we prove Theorem \ref{Thm_reg_holder}.
First, we clarify some notations.

\subsection*{Notation}Throughout this paper, we use $B_r$ to denote a ball of radius $r$ centered at the origin unless specified otherwise. We use $D$ to denote the derivative with respect to the ambient Euclidean metric, and $\nabla_g$ to denote the covariant derivative on the Lagrangian submanifold, with induced metric $g$.

\medskip

 We start by deriving properties of graphical smooth submanifolds of $\mathbb{R}^n\times \R^m$.

    \begin{lemma} \label{HolderImprovement}
Suppose that $F: \Omega  \subset  \R^n \to \R^m $ is $C^{\beta}$ where $\beta > \frac{1}{l+1}$ and $l \geq 1$ is an integer. If the graph of $F$ is a smooth submanifold of $\R^n \times \R^m$, then $F \in C^{\frac 1 l}_{\mathrm{loc}}( \Omega) .$ 
\end{lemma}

\begin{proof} Let us fix $K \Subset \Omega$. For the sake of contradiction, suppose that $F \notin C^{\frac 1 l } (K).$ Then there exists a sequence $\{x^{(i)}\}_{i=1}^\infty\subset K $ such that $x^{(i)} \to x \in K$ and 
$$ \limsup \frac{ |F(x^{(i)} )-F(x) |}{|x^{(i)} - x| ^{\frac1l}}   = \infty.$$
Without loss of generality, we assume that $x=0$ and $F(0) =0$.
Since the graph of $F$ is a smooth submanifold, the above necessarily implies that the graph of $F$ has a vertical tangent plane at $0$. Let $P$ denote this tangent plane and $\pi_x$ be the projection onto the $x$-plane. After rotations of the $x$ and $y$ coordinates, suppose that $\{\partial_{y_1}, \dots, \partial_{y_k}\}$ forms a basis for $\ker \pi_x|_P$ and $\{\partial_{x_{k+1}}, \dots, \partial _{x_n}\}$ forms a basis for $\pi_x ( P) .$ By the Implicit Function Theorem, the graph of $F$ can be represented as a graph of a smooth function $G(y_1, \dots, y_k, x_{k+1}, \dots , x_n ).$ 
We write \[x = ( x^\prime , x^{\prime \prime } ) \text{ and } y = ( y^\prime , y^{\prime \prime } )\] where $x^\prime =  (x_1, \dots, x_k)$, $x^{\prime \prime } = ( x_{k+1}, \dots, x_n)$, $y^\prime =  (y_1, \dots, y_k)$, and $y^{\prime \prime } = ( y_{k+1}, \dots, y_m)$ (note, it is possible that $k=m$ or $k=n$). 
Thus the graph of $F$ is locally the set
$$ \{ (x^\prime , x^{\prime \prime} , y^\prime , y^{\prime   \prime } ) \: | \: ( x^\prime, y^{\prime \prime })= G ( y^\prime , x^{\prime \prime } ) \} .$$
Now since the vectors $\partial_{y_1}, \dots, \partial_{y_k}$ are tangent to the graph of $F$, we get $D_{y^\prime } G(0) =0.$

We fix a unit vector $e$ in the $(y^\prime , x^{\prime \prime } ) $-plane and consider the curve $(te, G(te))$ for sufficiently small values of $t$. Next decompose $e$ into $x$ and $y$ components, $e = e_x + e_y$. Then we get
\[|e_y||t| \leq |y(t)|.\]

Assume $|e_y| \geq \frac 12.$ This yields $|y(t)| \sim |t|$. By the assumed H\"older continuity of $F$, we have
$$  |y(t) | \lesssim  | x(t)|^\beta .$$
By Taylor's Theorem, we get $|x(t) | \lesssim |t|^j$ for some integer $j$, which may depend on $e$. Let's suppose that, for some choice of $e$, $j> l$. Then we get
$$  |t| \lesssim |y(t) | \lesssim|x(t)|^\beta \lesssim   |t|^{(l+1) \beta }, $$
but this yields a contradiction since $(l+1) \beta >1.$ 

Therefore we must have $$|y(t)|^l \lesssim |t|^l \lesssim |x(t)|.$$ On the other hand, if $|e_x| > \frac 12$, then $|x(t) | \sim |t|$ and $|y(t) | \lesssim t $, which gives 
$$ |y(t)| ^l \lesssim |t|^l \lesssim |x(t)|.$$ 

This contradicts the assertion that 
$$ \frac{| y|}{|x|^{\frac 1 l } }  \to \infty $$
as $x \to 0$. Thus $F \in C^{ \frac 1 l } ( K) .$
\end{proof} 

Next, we establish smoothness of a Hölder continuous Hamiltonian stationary Lagrangian graph when the Hölder exponent exceeds $\frac{1}{2}$ and the phase satisfies $\Theta\geq(n-2)\frac{\pi}{2}+\delta$. The significance of this supercriticality of phase is that it permits a downward rotation of the gradient graph, under which the rotated submanifold can be shown to be smooth. By subsequently performing an upward rotation of the new submanifold, one recovers the smoothness of the original submanifold.  The method of performing such downward rotations for gradient graphs of $C^{1,1}$ functions with supercritical phase is explained in detail in Section 4 of \cite{ChenWarren}. In comparison, Section \ref{Sec_up_rot} of the present paper develops the corresponding upward rotation, established under considerably weaker assumptions on the phase. 
\begin{proposition} \label{Holder1/2to1}

Let $u \in C^{1, \beta}$ be a weak solution of \eqref{hstat0} on the unit ball $B_1\subset \mathbb{R}^n$ with $\beta \in (\frac12,1) $ and  $\Theta \geq (n-2)\frac{\pi}{2}+\delta$. Then $u$ is smooth inside $B_1$.
\end{proposition}

\begin{proof} 
Since the phase satisfies 
$\Theta \geq (n-2)\frac \pi 2 +\delta$, we can perform a downward rotation of the gradient graph $L_u=\{(x,Du(x))\: |\: x\in B_1\}\subset\R^n\times \R^n$ by an angle $\varphi  < \frac \delta n$. The new rotated coordinates are given by
$$ \begin{pmatrix} 
\bar x \\ \bar y 
\end{pmatrix} 
= 
\begin{pmatrix} 
\cos \varphi &  \sin \varphi 
\\ -\sin \varphi &\cos \varphi
\end{pmatrix} 
\begin{pmatrix} 
x \\ y
\end{pmatrix}.$$
By the computations in \cite[Section 4]{ChenWarren}, this yields a new potential $\bar u(\bar x)$, which is a $C^{1,1}$ solution of 
\[
\Delta_g \bar \Theta=0 \text{ on } B_{r}
\] 
where $B_{r}\subset\bar x(B_1)$ and
\[
\bar \Theta=\Theta -n\varphi>\Theta-\delta.
\] 
Note that although the above-referenced computations of \cite[Section 4]{ChenWarren} are presented for a $C^{1,1}$ potential, they apply verbatim to the case of a $C^{1,\beta}$ potential.

The rotated phase remains in the supercritical range, and hence by \cite[Theorem 1.2]{ChenWarren}, the new potential $\bar u$ is smooth. We point out that although our definition of a weak solution is slightly different than that in \cite{ChenWarren}, the proof of their result does not use the condition $u \in W^{2,n}$. Consequently, the gradient graph $ L_{\bar u}= \{ (\bar x , D \bar u (\bar x)) \: \vert \: \bar x \in B_r\} $ is a smooth submanifold of $\C^n$. Rotating it back up by the same angle $\varphi$, and using the fact that rotations preserve smooth manifolds, we conclude that the resulting submanifold, $L_u= \{ (x, D u (x) )\: \vert\: x\in B_{r'}\}$  where $B_{r'} \subset \bar x^{-1}(B_r)$, is also smooth.

 Now applying Lemma \ref{HolderImprovement} we get  $u \in C^{1,1}_{\mathrm{loc}}$. Then $u$ is smooth by \cite[Theorem 1.2]{ChenWarren}.
\end{proof} 

Next, we prove the following Lemma to push the H\"older exponent to $\frac{1}{2}$.

\begin{lemma}\label{Noninjective} If $F : \R^n \to \R^n$ is a smooth map with partial derivatives vanishing up to order $2k-1$ $(k \geq 1)$ and there exists a constant $c>0$ such that $|F(x)| \geq c |x|^{2k} $ for $x$ in $B_r$, then $F$ is not injective on any neighborhood of $0$. 
\end{lemma} 

\begin{proof} 
First we observe that since $|F(x)| \geq c |x|^{2k} $, the map $F$ has an isolated zero at $0$. If $F$ were injective in a neighborhood of $0$, its local degree must be $\pm 1$ \cite{Hatcher}. Let $Q$ denote the homogeneous degree $2k$ part of the Taylor expansion of $F$. Since $Q$ is of even degree, $Q(x) = Q(-x)$.

For $s < r$, consider the map $\Phi_s:S^{n-1} \to S^{n-1}$ given by \[\Phi_s(x) = \frac{ F(sx)}{|F(sx)|}.\]

Letting $s \to 0$, we see that $\Phi_s \to \frac{Q(x)}{|Q(x)|} $ uniformly, so $\Phi_s$ is homotopic to $ \frac{Q(x)}{|Q(x)|}$. By Sard's Theorem, there exists $y \in S^{n-1}$ such that $\left ( \frac{ Q(\cdot ) }{|Q(\cdot)| } \right )^{-1} ( \{ y \} ) $ is a discrete set. Since $Q$ is even, the preimage of $y$ has even cardinality and therefore $\frac{ Q(\cdot ) }{|Q(\cdot)| }$ has even degree. Since degree is homotopy invariant, $\Phi_s$ has even degree, and therefore $F$ has even local degree at $0$, contradicting injectivity. 
\end{proof} 

\begin{proposition}\label{holder_half} Let $u \in C^{1, \frac 12} $ be a weak solution of \eqref{hstat0} on the unit ball $B_1 \subset \R^n $ with $\Theta \geq (n-2) \frac \pi 2 + \delta $. Then $u$ is smooth inside $B_1$. \label{Holder1/2}
\end{proposition}

\begin{proof} As in the proof of Proposition \ref{Holder1/2to1}, we perform a downward rotation of the gradient graph  $L_u=\{(x,Du(x)) \: |\: x\in B_1\}\subset\R^n\times \R^n$. By \cite{ChenWarren}, we get that the gradient graph $L_u$ is a smooth submanifold of $\C^n$.

Let's assume $u$ is singular at $0$ and that $D u ( 0) =0.$ Let $\pi_x$ be the orthogonal projection onto the $x$-plane in $\C^n$. After rotating coordinates on the $y$-plane, we assume that $\{\partial_{y_1}, \dots , \partial_{y_k} \}$ $(k \geq 1)$ forms a basis for $\ker \pi_x |_P.$ After another coordinate change on the $x$-plane, we assume that $\{\partial_{x_{k+1} } , \dots , \partial_{x_n}\} $ forms a basis for $\pi_x ( P) .$ Then $P$ projects nondegenerately onto the plane spanned by the vectors $\partial_{y_1}, \dots , \partial_{y_k}, \partial_{x_{k+1} }, \dots , \partial_{x_n} $. So $L_u$ is a graph over this plane:
$$ L_u = \{ (y_1, \dots , y_k, x_{k+1 }, \dots , x_n, F ( y_1, \dots , y_k, x_{k+1 }, \dots , x_n)) \}$$
where $F : \R^n \to \R^n$. Since the vectors $\partial_{y_1}, \dots , \partial_{y_k}$ are tangent to $L_u$ at $0$, we get $\partial_{y_i} F (0) = 0 $ for $i = 1, \dots, k$. Therefore, $F$ has a Taylor expansion of the form 
\begin{equation}\label{TaylorExp}  F(y^\prime ,x^{\prime \prime }) = F_1(x^{\prime \prime }  ) + F_2(y^{\prime } ,x^{\prime \prime } ) + \dots \end{equation}
where  $F_i$ is the degree $i$ homogeneous term of the Taylor series of $F$.

Next we consider the map $F^\prime :\R^k \to \R^k$ given by projecting \[F(y_1, \dots, y_k, 0 , \dots , 0)  \] onto the $x_1, \dots , x_k$ directions. Since $D u \in C^{\frac 12 }$, points on $L_u$ satisfy $|y|^2 \leq C | x|$, which implies \[|F^\prime ( y_1 ,\dots, y_k) | \geq C |(y_1, \dots, y_k)|^2. \] Combining this with \eqref{TaylorExp} shows that $F^\prime$ is quadratic to leading order and verifies the hypothesis of Lemma \ref{Noninjective}. Therefore, $F^\prime$ is not injective. So there exist distinct points $(y_1^{(1)} , \dots, y_{k}^{(1)})$ and 
$(y_1^{(2)} , \dots, y_{k}^{(2)})$ such that $F^\prime (y_1^{(1)} , \dots, y_{k}^{(1)})=F^\prime(y_1^{(2)} , \dots, y_{k}^{(2)}) .$ This means that two different points of $L_u$ have the $x$-coordinate given by  \[(F^\prime (y_1^{(1)} , \dots, y_{k}^{(1)}), 0, \dots, 0),\] contradicting the assumption that $L_u$ is a graph. Therefore, $L_u$ does not have a vertical tangent plane in the interior of $B_1$. Thus $ u\in C^{1,1}_{\mathrm{loc}}$ and is therefore smooth in the interior by \cite[Theorem 1.2]{ChenWarren}.
\end{proof} 

\begin{proposition} \label{holder_third}
Let $u \in C^{1, \beta}$ be a weak solution of \eqref{hstat0}  on the unit ball $B_1\subset \mathbb{R}^n$  with $\beta \in ( \frac13 , \frac12)$ and  $\Theta \geq (n-2)\frac{\pi}{2}+\delta$. Then $u$ is smooth inside $B_1$.

\end{proposition} 

\begin{proof} The proof follows the same strategy as in Proposition \ref{Holder1/2to1}. Since the gradient graph $L_u$ is a smooth submanifold, Lemma \ref{HolderImprovement} implies that $u \in C^{1,\frac{1}{2}}_{\mathrm{loc}}$. Applying Proposition \ref{Holder1/2} then yields smoothness of $u$ in the interior.
\end{proof} 

Now we have all the ingredients to prove the main result of this section. 
\begin{proof}[Proof of Theorem \ref{Thm_reg_holder}.]
 Smoothness of $u$ is obtained by combining Propositions \ref{Holder1/2to1}, \ref{holder_half}, and \ref{holder_third}, after which the desired estimate follows from \cite[Theorem 1.2]{ChenWarren}.
\end{proof}

In fact, note that combining Lemma \ref{HolderImprovement} with Lemma \ref{Noninjective} and using the argument in the proof of Proposition \ref{Holder1/2}, we get the following. 

\begin{lemma} Suppose that $F : \R^n \to \R^n$ is $C^\beta $ where $\beta > \frac{ 1}{ 2l+ 1} $ and $l \geq 1$ is an integer. If the graph of $F$ is a smooth submanifold of $\R^n \times \R^n$, then $F \in C^{\frac{ 1} { 2l-1} }_{\mathrm{loc}}$. 
\end{lemma} 

This shows that singular H\"older continuous solutions of \eqref{hstat0} with phase $\Theta \geq  (n-2) \frac \pi 2  + \delta $ are precisely $C^{1,\frac 1 {2l+1}}$ for some integer $l \geq 1.$ In the next section, we show that each possible degree of H\"older regularity is achieved by a weak solution with hypercritical phase.

We conclude this section with an interesting observation on the H\"older continuity of the Lagrangian phase.

\begin{proposition}
    Let $u \in C^{1, \beta}$ be a semi-convex weak solution of \eqref{hstat0} on the unit ball $B_1\subset \mathbb{R}^n$. Then there exists a $\gamma$ such that the Lagrangian phase, $\Theta$, is $C^\gamma$ inside $B_1$. 
\end{proposition}

\begin{proof} 
 Since $u$ is semi-convex, we can rotate the gradient graph of $u$ down by an angle $\varphi$, which is dependent on the negative lower bound of $D^2u$, to obtain a new potential $\bar u \in C^{1,1}$. The new rotated coordinates are given by 
$$ \begin{pmatrix} 
\bar x \\ \bar y 
\end{pmatrix} 
= 
\begin{pmatrix} 
\cos \varphi &  \sin \varphi 
\\ -\sin \varphi &\cos \varphi
\end{pmatrix} 
\begin{pmatrix} 
x \\ y
\end{pmatrix}.$$ 

The new phase $\bar \Theta ( \bar x ) = \Theta ( x) - n \varphi $ is a weak solution of the equation $\Delta_{\bar g }\bar \Theta =0$, which is uniformly elliptic since $\bar u \in C^{1,1}$ as a function of $\bar x$. By the De Giorgi-Nash Theorem, $\bar \Theta$ is $C^{\alpha}$ in $\bar x$ for some $\alpha$. 

Since $D u \in C^\beta$ by assumption, we get
\begin{align*}
| \bar x_1 - \bar x_2 | & = | \cos \varphi ( x_1 - x_2 ) + \sin \varphi ( D u ( x_1 ) - D u (x_2 ) ) | \\
& \leq | x_1 - x_2 | + |  D u ( x_1 ) - D u (x_2 )  | \\
& \leq | x_1 - x_2 | + C_1 | x_1 - x_2 | ^ \beta \\
& \leq C_2 | x_1 - x_2 |^\beta .
\end{align*}

Then 
$$ \frac {| \Theta ( x_1 ) - \Theta ( x_2 ) | } {|x_1 - x_2 | ^{\alpha \beta } } \leq 
C_2^\alpha \frac{ | \bar \Theta ( \bar x_1 ) - \bar \Theta ( \bar x_2 ) |}{| \bar x_1 - \bar x_2 |^\alpha }, $$
so $\Theta$ is $C^{\alpha \beta }$ in $x$. 
\end{proof} 

\begin{remark}
Note that the above proposition is different from \cite[Corollary 4.6]{ChenWarren}, which proves that for semi-convex $C^1$ weak solutions of \eqref{hstat0}, the phase $\Theta$ is Hölder continuous with respect to the intrinsic metric on $L_u$. 
\end{remark}

\section{Singular solutions}\label{Sec_singsoln}


In this section, we prove Theorem \ref{Thm_sing_soln}. We begin with a Lemma establishing conditions under which a Lagrangian graph can be rotated upward while preserving the property of being a graph. We refer the reader to Section \ref{Sec_up_rot} for a detailed explanation of how such a rotation is performed on the gradient graph of a $C^{1,1}$ function. 

\begin{lemma}\label{rotatedgraph}
Let $w \in C^2(B_1)$ with $\lambda_{\max}(D^2 w) < 1$ everywhere except possibly on a discrete set, where $\lambda_{\max}$ denotes the largest eigenvalue of $D^2 w$. Then the submanifold of $\R^n \times \R^n$ obtained by rotating the gradient graph of $w$,
\[
L_w = \{ (x, Dw(x)) \;|\; x \in B_1 \},
\]
upward by an angle of $\tfrac{\pi}{4}$ is again a graph.
\end{lemma}

\begin{proof} We perform an upward rotation on $L_w$ by an angle of $\frac{\pi}{4}$. Then the rotated submanifold  has the following parametrization 
$$ \Phi ( x ) = \frac 1 {\sqrt 2} ( x - Dw(x) , x + Dw(x)).$$ 
We assume there are points $x, y \in B_1$ such that 
$$ \Phi (x) - \Phi (y)  = \sqrt  2(0,V)$$
is vertical, i.e., the image of $\Phi$ is not a graph. Then $$x-y = Dw(x)-Dw(y) = V.$$
But observe that
\begin{align*} 
\vert V\vert^2 & = \langle V, Dw(x) - Dw(y) \rangle \\
& = \int _0^1 \langle V, D^2 w ( y + t V ) V \rangle \: dt \\
& < \vert V\vert ^2 ,
\end{align*}
since $\lambda_{\max} ( D^2 w ( y + t V ) ) < 1 $ for almost every $t \in [0,1].$
This is a contradiction. Therefore, the rotated submanifold must be a graph. 
\end{proof}

Before constructing singular solutions to prove Theorem \ref{Thm_sing_soln},  we proceed with a useful Lemma which will allow us to take advantage of rotational symmetry when analyzing the Cauchy problem for \eqref{hstat0}.

Note that for the rest of this section, we will denote a point $x \in \R^n,$ by $x = (x^\prime, x_n)$ where $x^\prime \in \R^{n-1}$.

\begin{lemma}\label{Lem_axisymmetric}
Let $u$ be a smooth function defined on a neighborhood of the origin that is axisymmetric, i.e.,
$$ u(x',x_n) = v(|x'|,x_n). $$
Then the Hessian $D^2 u$ has the eigenvalue $\frac{v_r}{r}$ with multiplicity $n-2$, and its two remaining eigenvalues are precisely the eigenvalues of the $2 \times 2$ matrix $D^2 v$.
\end{lemma}

\begin{proof} Let $r = |x'|$. An elementary computation shows that
$$ D^2 u = \begin{pmatrix} 
\frac{v_r}{r} I_{n-1} + ( v_{rr} - \frac {v_r } r ) \frac{ x^\prime \otimes x^\prime }{r^2} & v_{rn} \frac{x^\prime}{r}\\
v_{rn} \frac{ (x^\prime )^T}{r} & v_{nn}
\end{pmatrix} .$$
For a vector $\xi \in \R^{n-1}$ with $\xi \perp x^\prime$,  we have
$$ D^2 u ( \xi,0 ) = \frac{v_r}{r} (\xi,0).$$
Also, observe that
\begin{align*} D^2 u ( \partial_r )& = v_{rr} \partial_r + v_{rn} \partial_n
\\ D^2 u ( \partial_n ) &= v_{rn} \partial_r + v_{nn} \partial_n.
\end{align*} 
This proves the claim. 
\end{proof}

Now we are ready to prove Theorem \ref{Thm_sing_soln}.

\begin{proof}[Proof of Theorem \ref{Thm_sing_soln}]
Fix a positive integer $k$.  We construct a solution $w$ to \eqref{hstat0} by solving the Cauchy problem 
\begin{equation*}
\begin{aligned} 
\Delta_g \Theta &=0,\\
w ( x^\prime , 0) & = \frac 12 | x^\prime |^2 - |x^\prime |^{2k+2} , \\
w_n ( x^\prime , 0) & = 0 ,\\
w_{nn}( x^\prime , 0)  & = -2 |x^\prime |^{2} ,\\
w_{nnn} ( x^\prime , 0) & = 0. 
\end{aligned}
\end{equation*}

The Cauchy-Kovalevskaya Theorem guarantees the existence of a unique analytic solution $w$ to the Cauchy problem. Observe that $w$ must be axisymmetric, since the Cauchy data is invariant under rotations in the $x^\prime$-plane. 

Furthermore, the Cauchy data ensures that $w$ has a Taylor expansion of the form 
$$
w(x) = \frac 12 r^2 - r^{2k+2} - x_n^2 r^2  + x_n^4W( r,x_n)
$$ where $W$ is an analytic function (and an even function in $r$). 

We have 
\begin{align}
\frac{w_r(x)}{r} & = 1 - (2k+2) r^{2k} - 2 x_n^2 + x_n^4 \frac{ W_r (r,x_n)}{r}, \label{wr/r} \\
w_{rr} (x) & = 1 - (2k+2)(2k+1) r^{2k} -2x_n^2+ x_n^4 W_{rr} (r,x_n ),\nonumber\\
w_{rn} (x) &= -4x_n r + O ( x_n^3)\nonumber. 
\end{align}
So near the origin, we get
\begin{equation*} \frac{w_r(x)}{r} \leq 1 - ( 2k+2)r^{2k} - x_n^2 .\label{lambda1_hd}\end{equation*}
Therefore $D^2w $ has eigenvalues $\lambda_1 , \dots, \lambda_{n-2}$ which reach a local 
maximum value of 1 at the origin by Lemma \ref{Lem_axisymmetric}. 
Moreover, observe that near the origin, we have
\begin{equation} \begin{aligned} \lambda_{n-1}( D^2 w ) &= w_{rr} + O ( w_{rn}^2 ) \\&= 1 - (2k+2) ( 2k+1 ) r^{2k}- 2 x_n^2  + C x_n^2 r^2 + O ( x_n^4 ) \\ & \leq 1  - (2k+2) ( 2k+1 ) r^{2k}-  x_n^2 .\label{lambda_n-1_hd} \end{aligned} \end{equation} 
So $\lambda_{n-1}$ also reaches a local maximum value of $1$ at the origin. The final eigenvalue $\lambda_n$ is $0$ at the origin.

Since $n-1$ eigenvalues of $D^2 w$ have local maximum values of $1$ at the origin, we can rotate the gradient graph of $w$ so that the tangent plane at the origin becomes vertical. 

Using Lemma \ref{rotatedgraph}, we rotate the gradient graph $L_w$ upward by an angle of $\frac \pi 4 $ to obtain a new Lagrangian graph which is the gradient graph of a new potential $u^{(n,k)}$. After a rescaling, the new gradient graph, $L_{u^{(n,k)}}$, is the image of the following parametrization
\begin{equation}\begin{aligned} 
\Phi ( x) &= (x - Dw(x), x + Dw(x)) \\
& = ( [(2k+2) |x^\prime | ^{2k}+2x_n^2 ] x^\prime , [1 + 2 | x^\prime |^2 ]x_n,\\ & \quad  [2-(2k+2) |x^\prime | ^{2k}-2x_n^2 ]x^\prime , [1-2 |x^\prime |^2] x_n ) + O ( x_n^3). \label{rotated_graph_higher}
\end{aligned} \end{equation}
Set $\bar x = x - Dw(x)$ and $\bar y = x + Dw(x).$

To show that the image of $\Phi$ is the graph of a $C^{\frac{1}{2k+1}}$ function, we must show that $\bar y$ is $C^{\frac{1}{2k+1}} $ as a function of $\bar x$. If not, then there would be a point $\bar x_0$ such that 
$$ \limsup_{ \bar x \to \bar x_0} \frac{ | \bar y ( \bar x ) - \bar y ( \bar x_0 ) |}{|\bar x - \bar x_0|^{\frac{1}{2k+1}}} = \infty .$$

Necessarily, $\bar y $ is not differentiable at $\bar x_0$. By our construction, the only point where $\bar y$ is not differentiable is the origin. Therefore, it suffices to establish H\"older continuity of $\bar y$ at the origin. Since $\bar y = - \bar x + 2 x ( \bar x ) $, it is enough to show that $x $ is a $C^{\frac 1 {2k+1} }$ function of $\bar x$ at the origin. From \eqref{rotated_graph_higher}, we get 
$$ |\bar x | ^2 \geq ( 2k + 2 ) |x^\prime | ^{4k+2 } + 4 | x^\prime | ^4 x_n^2 \geq C |x|^{4k+2} $$
near the origin. 
Therefore $u^{(n,k)}  \in C^{\frac{1}{2k+1}}$, as desired.
Considering the image of the $x_1$ axis under $\Phi$ shows that $Du$ has no higher H\"older continuity. 

Since $Du^{(n,k)} \bar x = \bar y( \bar x) $, to show that $u^{(n,k)} \in W^{2,1}$, it is enough to show that $x$ is a $W^{1,1}$ function of $\bar x$. 
We have 
$$ D_{ \bar x } x = ( I - D^2 w ) ^{-1} .$$
Thus, it suffices to show that 
$$ \int_{B_\rho} | (I - D^2 w ) |^{-1}  \: d\bar x < \infty.$$
Using the standard change of variables formula, we get
$$ \int_{B_\rho} | (I - D^2 w ) ^{-1}| \: d\bar x = \int _{x^{-1} (B_\rho)} | (I - D^2 w )| ^{-1}  \det ( I - D^2 w ) \: dx. $$

By equations \eqref{wr/r}-\eqref{lambda_n-1_hd}, there are positive constants $C_1, C_2, C_3 ,C_4$ such that 
$$ | (I - D^2 w)^{-1} | \sim ( 1 - \lambda_{n-1} ( D^2 w ) ) ^{-1} \leq \frac{ 1}{ C_1 | x^\prime |^ {2k} + C_2 x_n ^2 } $$ and 
$$ 1- \lambda_i ( D^2 w ) \leq C_3 |x^\prime | ^ {2k} + C_4 x_n^2 $$ for $1 \leq i \leq n-2.$ 
Note that $\frac{ C_3 |x^\prime |^{2k} + C_4 x_n^2 } { C_1 |x^\prime |^{2k} + C_2 x_n^2 } $ is bounded, and hence integrable, so $u^{(n,k)} \in W^{2,1}.$

Similarly, to show that $\sqrt { g } $ is integrable, note that $\sqrt {  g }  \sim \det ( I - D^2 w )^{-1} $. After changing variables from $\bar x$ back to $x$, the integrand becomes bounded. 

Since $w$ is a smooth solution of \eqref{hstat0}, $\Theta$ is a smooth function on $L_w$. Therefore, after performing the rotation, the phase $ \bar \Theta $ is still a smooth function on the submanifold $L_{u^{(n,k)}}$. Moreover, $\bar \Theta ( \bar x ) = \Theta (x) + n \frac \pi 4 .$ Therefore 
$$ D_{\bar x} \bar \Theta =  (I - D^2 w )^{-1}  D_x \Theta .$$
Hence to show that $\bar \Theta \in W^{1,2} ( B_1)$, it is enough to show that 
$$ \int _{B_1}|(I - D^2 w )^{-1} |^2 \: d \bar x < \infty.$$
Changing variables back to $x$, the integrand becomes bounded again. 

Finally, we see that $u^{(n,k)} $ satisfies \eqref{hstat0} in the integral sense because $L_{u^{(n,k)}} $ is a smooth submanifold and the phase is a smooth harmonic function on $L_{u^{(n,k)}} $. 
\end{proof} 

\begin{remark} 
Every step of the above proof carries over to dimension $2$ except that the phase $\Theta$ lies in $W^{1,1}$ rather than $W^{1,2}$,
i.e., we would get a weak solution if condition (3) in Definition \ref{def1} is replaced by $\Theta \in W^{1,1}(\Omega)$. Although this alternative requirement can allow for a meaningful weak formulation in the Hamiltonian stationary setting, we adopt the condition $\Theta \in W^{1,2}(\Omega)$ to remain consistent with the standard notion of weak solutions in the literature. Note that this also stands in contrast with the theory of special Lagrangians, where in two dimensions any $C^0$ viscosity solution of the special Lagrangian equation is known to be analytic for all phases $|\Theta| \ge 0$ as shown by Warren-Yuan \cite{WY2d}.
\end{remark}

We conclude this section by characterizing the Sobolev regularity of the weak solutions constructed above.

\begin{lemma}\label{integral}
For $q, k >0$, the function $\frac{ 1}{(|x^\prime|^{2k} + |x_n|^2)^q } \in L^1_{\mathrm{loc}} ( \R^n)$ if and only if $$ q< \frac 12 + \frac{ n-1}{2k} .$$

\end{lemma} 

\begin{proof} 

Let $r = |x'|$ and $s = |x_n|$.  Using polar coordinates in the $x'$-plane, we need to determine when
$$ \int_0^1 \int_0^1 \frac{r^{n-2}} {( r^{2k} + s^2 )^q} \: dr \: ds < \infty.$$
Split the integral into the regions where $s\leq r^k $ and $s\geq r^k $.

In the region $s \leq r^k$, we have 
\begin{align*} \int_0^1 \int _0^{r^k} \frac{ r^{n-2}} { ( r^{2k} + s^ 2 ) ^q } \: ds \: dr &\leq \int_0^1 \int _0^{r^k} r^{n-2-2kq } \: ds \: dr\\
& = \int_0^1 r^{n-2-2kq+k} \: dr. 
\end{align*}

The final integral converges if and only if $n-2-2kq + k > -1$, or equivalently, $q < \frac 12 + \frac{n-1}{2k} .$

In the region $s\geq r^k $, we have 
\begin{align*} \int_0^1 \int _0^{s^{\frac1k}} \frac{ r^{n-2}} { ( r^{2k} + s^ 2 ) ^q } \: dr \: ds &\leq \int_0^1 \int _0^{s^{\frac 1k} } r^{n-2 } s^{-2q}  \: dr \: ds\\
& = \frac{1}{n-1} \int_0^1 s^{\frac{n-1}{k}-2q } \: dr. 
\end{align*}
Again, the last integral is finite if and only if $\frac{ n-1}{k} - 2q > -1$, or equivalently $q < \frac 12 + \frac {n-1}{2k}.$ This shows that the condition on $q$ is sufficient to ensure $\frac{ 1}{(|x^\prime|^{2k} + |x_n|^2)^q } \in L^1_{\mathrm{loc}} ( \R^n)$. 

To show the condition is necessary, again consider the region where $s\geq r^k$. Here we have 

\begin{align*} \int_0^1 \int _0^{r^k} \frac{ r^{n-2}} { ( r^{2k} + s^ 2 ) ^q } \: ds \: dr \geq \int _0^1 \int _0^{s^{\frac 1k} } r^{n-2 - 2kq } \: dr \: ds.
\end{align*} 
The inner integral converges if and only if $q < \frac { n-1} {2k} $, and in this case we have 
\begin{align*} 
\int _0^1 \int _0^{s^{\frac 1k} } r^{n-2 - 2kq } \: dr \: ds = \frac{ 1}{n-1-2kq } \int_0^1 s^{\frac{ n-1 - 2kq }{k}} \: ds 
\end{align*} 
which converges if and only if $q < \frac 12 + \frac {n-1}{2k}.$
\end{proof}

\begin{proposition} The solution $u^{(n,k)}$ constructed above is in $W^{2,p}$ for $p < n - \frac 12 + \frac {n-1}{2k} .$ In particular, $u^{(n,k)} \in W^{2,n} $ for $k < n-1$. 

\begin{proof} 
By the proof of Theorem \ref{Thm_sing_soln}, $u^{(n,k)} \in W^{2,p} $ whenever 
$$ \int_{B_\rho} | (I - D^2 w ) ^{-1}|^p \: d\bar x < \infty. $$
After changing variables and applying equations \eqref{wr/r}-\eqref{lambda_n-1_hd}, convergence of the integral is equivalent to convergence of 
$$ \int _{B_\rho } \frac{ 1}{(|x^\prime |^{2k} + x_n^2 )^{p - n + 1}} \: dx .$$
Applying Lemma \ref{integral} shows that this integral converges if and only if $ p - n+1 <  \frac 12 + \frac { n -1 }{2k} .$
\end{proof} 

\end{proposition}

\section{Regularity of Lipschitz continuous HSL graphs}\label{Sec_Lipschitz}

In this section, we prove  Theorem \ref{Thm_reg_Lipschitz} by first establishing an upward rotation technique.

\subsection{Upward Rotation} \label{Sec_up_rot} We provide a detailed explanation of how an upward rotation is performed on the gradient graph of $u\in C^{1,1}$. 

\begin{proposition}\label{upward}
    Let $u$ be a $C^{1,1}$ function on $B_R(0)\subset \re^n$. Then the Lagrangian submanifold $ L_u=(x,Du(x))\subset \re^n\times \re^n$ can be represented as a gradient graph $(\bar{x},D\bar{u}(\bar{x}))$ of the new potential \begin{equation}
    \bar{u}(x)=u(x)-\sin\gamma\cos\gamma\frac{\vert Du\vert ^2-\vert x\vert^2}{2}-\sin^2\gamma (x\cdot Du(x)) \label{ubar}
    \end{equation}
    such that in the $\bar x$ coordinates the rotated gradient graph is Lipschitz.
\end{proposition}

\begin{proof} Let $K = \sup_{B_R} |D^2 u |$. Since 
\[
\sum_{i=1}^n \arctan \lambda_i<n\pi/2
\]
set \[\gamma = \frac{ 1}{ 2 } ( \frac{ \pi } 2  - \arctan K ). \] 
Define the new rotated coordinates on $\C^n$ by 
$$ \begin{pmatrix} 
\bar x \\ \bar y 
\end{pmatrix} 
= 
\begin{pmatrix} 
\cos \gamma & - \sin \gamma 
\\ \sin \gamma &\cos \gamma  
\end{pmatrix} 
\begin{pmatrix} 
x \\ y
\end{pmatrix}. $$
For any two points $(x_1, D u ( x_1) ) $, $ (x_2, D u ( x_2) )$, where $x_1,x_2\in B_R$, we estimate 
$$ | \bar x_1  - \bar x_2 | ^2  = | \cos \gamma ( x_1  - x_2 ) - \sin \gamma ( D u ( x_1) - D u ( x_2 ) )|^2 .$$
Let $w(x) = u(x) - \frac K 2 | x|^2 $. Then 
\begin{equation*} 
\begin{aligned} 
\frac{ | \bar x_1 - \bar x_2  |^2 }{ \cos^2 \gamma }  
& = | x_1 - x_2 - \tan \gamma ( D u ( x_1) - D u ( x_2 ) ) |^2 \\
& = | x_1 - x_2- K \tan \gamma ( x_1 - x_2 )  - \tan \gamma ( D w ( x_1) - D w ( x_2 ) ) |^2\\
& = ( 1 - K \tan \gamma )^2 |x_1-x_2|^2 + \tan ^2 \gamma | D w ( x_1) - D w ( x_2)|^2 \\
& \quad -2 \tan \gamma ( 1 - K \tan \gamma ) ( x_1 - x_2 ) \cdot ( D w ( x_1 ) - D w ( x_2)) .
\end{aligned} 
\end{equation*} 
Since $|D^2 u |\leq K, $ $w$ is concave, yielding \[( x_1 - x_2 ) \cdot ( D w ( x_1 ) - D w ( x_2))  \leq 0.\]
Observe that as $\gamma < \frac \pi 2 - \arctan K $, we get  \[\tan \gamma < \cot ( \arctan K ) = \frac 1 K .\] Therefore 
\begin{equation}
    | \bar x_1 - \bar x_2 | \geq (\cos \gamma - K \sin \gamma )| x_1 - x_2| \label{lambda},
 \end{equation}
which shows that the rotated submanifold is still a graph, in particular, $(\bar x, D \bar u (\bar x))$. Note
\begin{equation*} 
\begin{aligned}
    \sum_{i=1}^n \bar y_i \: d\bar x_i &= \sum_{i=1}^n (\sin \gamma x^i +\cos \gamma u_i(x))(\cos \gamma \: dx^i-\sin\gamma u_{ij}(x)\: dx^j\\
     &= \sum_{i=1}^n (\sin\gamma \cos \gamma x^i \: dx^i+ \cos^2\gamma u_i(x)\: dx^i-\sin^2\gamma x^iu_{ij}(x)\: dx^j\\
     &-\cos\gamma \sin\gamma u_i(x)u_{ij}(x)\: dx^j)\\
     &= \sin\gamma \cos \gamma D\left (\frac{|x|^2}{2}\right )+\cos^2 \gamma Du(x)\\
     & \quad -\sin^2\gamma(D(x\cdot Du(x))-Du(x))-\cos\gamma\sin\gamma D\left (\frac{|Du(x)|^2}{2}\right)\\
     &=Du(x)-\sin\gamma\cos\gamma D\left (\frac{|Du(x)|^2-|x|^2}{2}\right )-\sin^2\gamma D(x\cdot Du(x)).
\end{aligned} 
\end{equation*}

So,
\[
\bar u (x)=u(x)-\sin\gamma\cos\gamma\frac{\vert Du\vert ^2-\vert x\vert^2}{2}-\sin^2\gamma (x\cdot Du(x)).
\]

Observe that the rotated potential $\bar u$ is in $C^{1,1} $ 
\begin{equation*}
    \begin{aligned}
        |D\bar u(\bar x_1)-D\bar u(\bar x_2)|  
        & \leq  \frac{ | \sin \gamma  ( x_1 - x_2 ) + \cos \gamma ( D u ( x_1) - D u ( x_2 ) ) | } { ( \cos \gamma  - K \sin \gamma ) |x_1 - x_2  |}|\bar x_1-\bar x_2|\\
        &\leq \bigg(   \frac{\sin \gamma }{ \cos \gamma - K \sin \gamma } + \frac{ 1}{ \cos \gamma - K \sin \gamma } \| u \| _{ C^{1,1} }\bigg) |\bar x_1-\bar x_2|.
    \end{aligned}
\end{equation*}
This shows that the rotated gradient graph is Lipschitz.
\end{proof}

\subsection{Proof of Theorem \ref{Thm_reg_Lipschitz}.}
Now we are ready to prove Theorem \ref{Thm_reg_Lipschitz}.
\begin{proof}[Proof of Theorem \ref{Thm_reg_Lipschitz}.]
    We prove the smoothness of $u$ on a neighborhood, $U_0$, of the set $\{ x \in B_{\frac 12 } \: \vert \: \Theta ( x) \geq (n-2 ) \frac \pi 2 \} .$ By symmetry (or a downward rotation argument) it will follow that the result is also true on a neighborhood of the set $\{x \in B_{\frac12} \: \vert \: \Theta ( x) \leq - ( n-2 ) \frac \pi 2 \}.$
    
    We perform an upward rotation of the gradient graph $\{(x, Du(x))\: \vert\:  x \in B_1\}$ by an angle of $\gamma$ as outlined in the above Proposition. Let $\Omega=\bar x (U)$ where $U = \{ x \in B_{\frac 34} \: \vert \: \Theta ( x) \geq (n-2) \frac \pi 2 - (n-1)\gamma\}.$
    
    So, $\bar u$ satisfies the Lagrangian mean curvature equation in $\Omega$,
    \begin{equation}
        \sum_{i=1}^n \arctan \bar \lambda_i=\bar \Theta(\bar x), \label{rs}
    \end{equation} 
    where $\bar\lambda_i$ are the eigenvalues of $D^2 \bar u$. Since the upward rotation increases the phase, we have
     \[
    \sum \arctan \bar \lambda_i  = \sum \arctan \lambda_ i + n \gamma . \]
In particular, this ensures that $\bar \Theta \geq ( n-2) \frac \pi 2 + \gamma $ in $\Omega$.

    The De Giorgi-Nash Theorem applied to \eqref{hstat0}, gives $\Theta\in C^{\alpha}(B_1)$, for some $\alpha$, with estimate
    \begin{equation}
        \|\Theta\|_{C^{\alpha}(B_{3r'/4})}\leq C(\| D^2 u\|_{L^{\infty}(B_{r'})})\label{DGN} 
    \end{equation} where $B_{r'}\subset U.$

   Now denoting \[
    \lambda=\frac{1}{\cos\gamma-K\sin\gamma},
    \]
    we verify that $\bar \Theta \in C^{\alpha}(B_{\bar r}) \subset \Omega$ where $\bar r= r'/2\lambda$ and $B_{r'}\subset U$:
    \[
    \frac{|\bar\Theta(\bar x_1)-\bar\Theta(\bar x_2)|}{|\bar x_1-\bar x_2|^\alpha}\leq  \frac{|\Theta( x_1)-\Theta( x_2)|}{| x_1- x_2|^{\alpha}}\lambda^{\alpha}\leq \|\Theta\|_{C^{\alpha}(B_{r'})}\lambda^{\alpha}.
    \]
    After modifying the uniformly elliptic second order equation \eqref{rs} to a concave one (see \cite[pg 347]{ChenWarren} or \cite[Lemma 2.2]{CPW}), and applying \cite[Corollary 1.3]{CC01}, we get $\bar u\in C^{2,\alpha}(B_{\frac{\bar r}{2}})$ with estimates
    \[
    |D^2\bar u(\bar x)-D^2\bar u(0)|\leq C\bigg( \|\bar u\|_{L^{\infty}(B_{\bar{r})}}+\|\bar \Theta\|_{C^{\alpha}(B_{\bar{r}})}\bigg) |\bar x |^\alpha 
    \] where $C$ depends on $K$ and the dimension $n$.
    Now, we rotate the gradient graph $\{ (\bar x, D\bar u(\bar x)\:|\: \bar x \in B_{\frac{\bar r}{2}}\}$ down by an angle of $\gamma$ to the original coordinates: 
    $$ \begin{pmatrix} 
x \\  y 
\end{pmatrix} 
= 
\begin{pmatrix} 
\cos \gamma &  \sin \gamma 
\\ -\sin \gamma &\cos \gamma  
\end{pmatrix} 
\begin{pmatrix} 
\bar x \\ \bar y
\end{pmatrix}. $$

Note that 
\[
\frac{dx}{d\bar x}=\cos\gamma I_n+\sin\gamma D^2_{\bar x}\bar u(\bar x)=\bigg(\frac{d\bar x}{d x}\bigg)^{-1}\geq \frac{1}{\cos\gamma+\sin\gamma K}I_n=\tilde\lambda I_n.
\]

Using the above to connect the oscillations of the original Hessian to the rotated Hessian, for $x\in B_{r_0}$ where $B_{r_0}\subset \bar x^{-1}(B_{\frac{\bar r}{2}})$, we get

\[
|D^2u(x)-D^2u(0)|\leq C \frac{1}{\tilde \lambda^{2+\alpha}}\bigg(\|\bar u\|_{L^{\infty}(B_{\bar r})}+\|\Theta\|_{C^{\alpha}(B_{\bar r})}\bigg)|x|^{\alpha}.
\]

Recalling \eqref{ubar} and the inequality connecting the $C^{\alpha}$ norms of $\Theta$ and $\bar \Theta$, we get

\begin{align*}
    |D^2u(x)-D^2u(0)| &\leq C\frac{1}{\tilde \lambda^{2+\alpha}}\bigg(\| u\|_{L^{\infty}(B_1)}+\| Du\|_{L^{\infty}(B_1)}\\
    &+ \frac{1}{2}[1+\| Du\|^2_{L^{\infty}(B_1)}]+ \lambda^{\alpha}\bar r\|\Theta\|_{C^{\alpha}(B_1)}\bigg)|x|^{\alpha}.
\end{align*}

Combining the above with \eqref{DGN}, and applying Schauder theory (see \cite[pg 347-349]{ChenWarren} for a detailed explanation) and a covering argument, we get the desired estimate. 
\end{proof}

\bibliographystyle{amsalpha}
\bibliography{AMS}

\end{document}